\documentclass[12pt]{article}

\usepackage{graphicx}
\usepackage{amsmath,amsthm}
\usepackage{amsfonts}
\usepackage{amssymb}

\theoremstyle{plain}
\newtheorem{gessel}{Theorem}[section]
\newtheorem{pathmatrix}[gessel]{Lemma}
\newtheorem{launois1}[gessel]{Theorem}
\newtheorem{launois2}[gessel]{Theorem}
\newtheorem{edgecommutivity}[gessel]{Lemma}
\newtheorem{disjointpath}[gessel]{Lemma}
\newtheorem{splitpath}[gessel]{Lemma}
\newtheorem{cauchongraphproperties}[gessel]{Proposition}
\newtheorem{maincorollary}[gessel]{Theorem}
\newtheorem{maintheorem}[gessel]{Theorem}
\newtheorem{goodearlletzter}[gessel]{Theorem}
\newtheorem{consecutivepaths}[gessel]{Lemma}

\theoremstyle{definition}
\newtheorem{cauchondiagram}[gessel]{Definition}
\newtheorem{cauchongraph}[gessel]{Definition}
\newtheorem{quantumalgebra}[gessel]{Definition}
\newtheorem{pathsystem}[gessel]{Definition}
\newtheorem{pathmatrixdef}[gessel]{Definition}
\newtheorem{algorithm}[gessel]{Algorithm}
\newtheorem{graphnotation}[gessel]{Notation}

\newtheorem{basicnotation}[gessel]{Notation}
\newtheorem{Talgebra}[gessel]{Definition}
\newtheorem{qtrans}[gessel]{Remark}

\newtheorem{example}[gessel]{Example}
\newtheorem*{claim}{Claim}

\title{A Graph Theoretic Method for Determining Generating Sets of Prime Ideals in Quantum Matrices}

\author{Karel Casteels\footnote{Author supported by the National Sciences and Engineering Research Council of Canada.}  \\
Department of Mathematics \\
Simon Fraser University \\
8888 University Dr. \\
Vancouver BC, Canada\\
V5A 1S6}

\begin{document}
\maketitle
\date

\begin{abstract}
We take a graph theoretic approach to the problem of finding generators for those prime ideals of $\mathcal{O}_q(\mathcal{M}_{m,n}(\mathbb{K}))$ which are invariant under the torus action ($\mathbb{K}^*)^{m+n}$. Launois \cite{launois3} has shown that the generators consist of certain quantum minors of the matrix of canonical generators of $\mathcal{O}_q(\mathcal{M}_{m,n}(\mathbb{K}))$ and in \cite{launois2} gives an algorithm to find them. In this paper we modify a classic result of Lindstr\"{o}m \cite{lind} and Gessel-Viennot~\cite{gv} to show that a quantum minor is in the generating set for a particular ideal if and only if we can find a particular set of vertex-disjoint directed paths in an associated directed graph. 
\end{abstract}

\section{Introduction}
Let $\mathbb{K}$ be a field of characteristic zero. Let $\mathcal{A}=\mathcal{O}_q(\mathcal{M}_{m,n}(\mathbb{K}))$ be the quantized coordinate ring of $m\times n$ matrices over $\mathbb{K}$ (informally referred to as the algebra of $m\times n$ quantum matrices).

Recent attention has focused on understanding the prime spectrum of $\mathcal{A}$. Goodearl and Letzter \cite{gl} have developed a powerful stratification theory that allows one to restrict attention to those prime ideals that are held stable under the action of an algebraic torus $\mathcal{H}=(\mathbb{K}^*)^{m+n}$. Call the set of such ideals $\mathcal{H}$-spec($\mathcal{A}$).

Cauchon \cite{cauchon1} applied his deleting derivations algorithm to $\mathcal{A}$ and obtained a bijection between $\mathcal{H}$-spec($\mathcal{A}$) and a set of combinatorial objects which have come to be called \emph{Cauchon diagrams}. The notion of a Cauchon diagram has recently been extended by M\'eriaux~\cite{mer2} to other classes of quantum algebras, however, in the context of $m\times n$ quantum matrices, a Cauchon diagram consists of an $m\times n$ grid of squares, each square coloured black or white so that for any black square, either every square above it or every square to its left is also black. An example appears in Figure~\ref{example}. 

Launois \cite{launois3} further developed these ideas and was able to prove a conjecture of Goodearl and Lenagan \cite{gl4} that the ideals of $\mathcal{H}$-spec($\mathcal{A})$ are generated by quantum minors of the canonical matrix $X_\mathcal{A}$ of generators of $\mathcal{A}$ (see Definition~\ref{quantumalgebra}). Furthermore, Launois \cite{launois2} gave an algorithm that explicitly determines the quantum minors in question. The first step of this algorithm is to calculate a certain matrix $T$ with entries in a McConnell-Pettit algebra. The next step is to determine the vanishing minors of this matrix. 

The main contribution of this paper is Theorem~\ref{main} which essentially states that Launois' algorithm is equivalent to finding certain sets of disjoint paths in a Cauchon diagram. Roughly speaking, we show that a $k\times k$ submatrix of $T$ has non-vanishing quantum determinant if and only if we can find a corresponding set of $k$ disjoint paths in the Cauchon diagram. For example, Figure~\ref{example} gives an example of a $4\times 5$ Cauchon diagram with two paths drawn over top; the existence of these paths implies that the quantum minor of $T$ indexed by $\{1,2\}\times\{1,2\}$ is non-zero, so that the corresponding quantum minor of $X_\mathcal{A}$ does \emph{not} belong to the $\mathcal{H}$-prime ideal associated to this Cauchon diagram. This method was inspired by an old result of Lindstr\"{o}m \cite{lind} (often attributed to Gessel-Viennot \cite{gv}). We will call this result Lindstr\"om's Lemma (see \cite{proofsbook} for an excellent exposition). A major component of our work (Theorem~\ref{gv}) is the proof of a $q$-analogue of a special case of Lindstr\"om's Lemma.
\begin{figure}
\centering
\includegraphics[height=4.5cm]{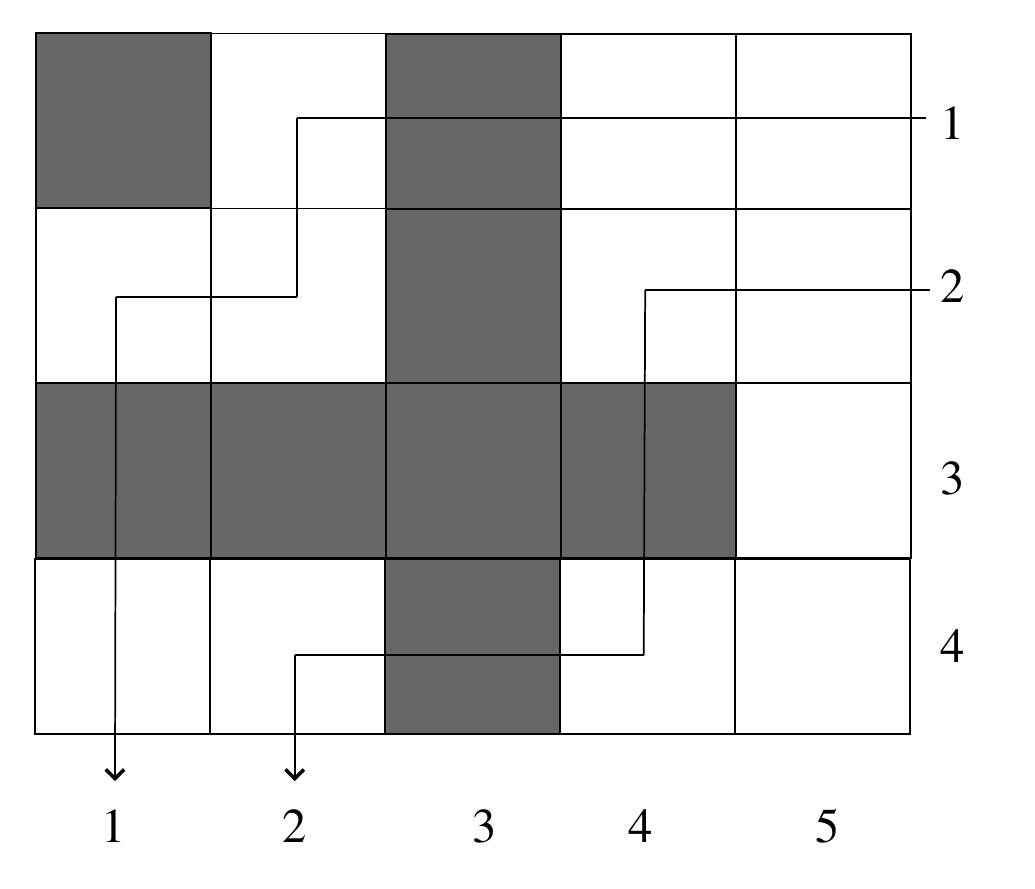}
\caption{A Cauchon diagram together with two disjoint paths from rows $1$ and $2$ to columns $1$ and $2$.} \label{example}
\end{figure}

Cauchon diagrams essentially appeared independently in the work of Postnikov~\cite{postnikov} on totally nonnegative Grassmannians where they are there called \reflectbox{L}-\emph{diagrams} (also sometimes written as \emph{Le-diagrams}). His work implies a correspondence between Cauchon diagrams and the collection of totally non-negative matrices over $\mathbb{R}$ (that is, matrices all of whose minors are non-negative). The connection between Postnikov's work and the ideal structure of $\mathcal{A}$ has recently been developed by Goodearl, Launois and Lenagan~\cite{gll2, gll}. In view of this and the results of this paper, it is perhaps not surprising that Talaska~\cite{tal} has independently been able to give an explicit description of the correspondence between Postnikov's \reflectbox{L}-diagram and totally-nonnegative-matrices using the classic version of Lindstr\"om's Lemma. 

Finally, we note that by a result of Launois~\cite{launois1} (see also~\cite{gll2,gll}), the poset structure of $\mathcal{H}$-spec($\mathcal{A}$) is order-isomorphic to a subset of $(n+m)$-permutations ordered with the Bruhat order. Yakimov~\cite{yakimov} has recently developed an alternate approach to determine generators for $\mathcal{H}$-primes based on such permutations. These sets also consist of quantum minors but, in contrast to our method, do not necessarily include every quantum minor in the given $\mathcal{H}$-prime.

\section{Background}
\subsection{Basic Definitions}
In this section we give the definitions which are of relevance to us and outline some of the basic past results in this area of study.

\begin{basicnotation} \label{basicnotation}
Below is a list of notation and conventions which will be used in this work.
\begin{itemize}
\item If $n$ is a positive integer, then  $[n]:=\{1,2,\ldots,n\}$.
\item As we will be working in a noncommutative algebra, we clarify that the standard product notation of a set of elements, say $\prod_{i=1}^n z_i$ is the ordered product, from \emph{left to right} of the $z_i$, i.e.,
$$\prod_{i=1}^n z_i = z_1z_2\cdots z_n.$$
\item If $\sigma$ is a permutation acting on a finite set of integers $X$, then the \emph{length} $\ell(\sigma)$ denotes the total number of pairs $i,j\in X$ such that $i<j$ and $\sigma(i)>\sigma(j)$. Such a pair is called an \emph{inversion}. 
\item If $M$ is an $m\times n$ matrix with $I\subseteq [m]$ and $J\subseteq[n]$, denote by $M[I,J]$ the submatrix of $M$ whose rows are indexed by $I$ and columns indexed by $J$. 
\item The \textit{quantum determinant}, or \textit{$q$-determinant} of a $k\times k$ matrix $M$ with respect to $q\in\mathbb{K}$ is the quantity $${\det}_q(M):= \sum_{\sigma\in S_k} (-q)^{l(\sigma)}\prod_{i=1}^k M[i,\sigma(i)].$$
\item We will use the standard partial ordering on pairs of integers by setting $(i,j)\leq (s,t)$ if and only if $i\leq s$ and $j\leq t$.
\item We also totally order the set $([m]\times [n])\cup (m,n+1)$ by setting $(i,j)\preceq (s,t)$ if and only if either $i<s$ or $i=s$ and $j \leq t$. This is called the \textit{lexicographic} ordering. 
\item Let $(i,j)\in ([m]\times [n])\cup(m,n+1)$. If $(i,j)\neq (1,1)$, we denote by $(i,j)^-$ the greatest element less than $(i,j)$ in the lexicographic ordering. If $(i,j)\neq (m,n+1)$, we denote by $(i,j)^+$ the least element which is greater than $(i,j)$ in this ordering.

\end{itemize}
\end{basicnotation}

We assume the reader is familiar with the elementary definitions of graphs and directed graphs but we refer the uninitiated reader to Wilson~\cite{wilson} for these details.

The algebraic structure of interest to us is defined as follows.

\begin{quantumalgebra} \label{quantumalgebra}
Fix $q\in\mathbb{K}^*$ and two positive integers $m$ and $n$. The $\mathbb{K}$-algebra $\mathcal{A}=\mathcal{O}_q(\mathcal{M}_{m,n}(\mathbb{K}))$ is the quantized coordinate ring of $m\times n$ matrices over $\mathbb{K}$ (informally referred to as ``$m\times n$ quantum matrices''). 
In other words, $\mathcal{A}$ is the $\mathbb{K}$-algebra generated by the $mn$ indeterminants $x_{i,j}$ (the \emph{canonical generators}) which satisfy the following relations. 
Consider the $m\times n$ matrix $X_\mathcal{A}$ where $X_\mathcal{A}[i,j] = x_{i,j}$. If $m\geq 2$ and $n\geq 2$, then for any $2\times 2$ submatrix
$\left[\begin{array}{cc}
a & b \\
c & d
\end{array}\right]$
 of $X_\mathcal{A}$, the following hold:
\begin{enumerate}
	\item $ab= q\,ba$ and $cd= q\,dc$,
	\item $ac= q\,ca$ and $bd= q\,db$,
	\item $bc=cb$,
	\item $ad - da = (q-q^{-1})\,bc$. \label{qa4}
\end{enumerate}
If $m=1$, then for $j>k$, we set $x_{1,j}x_{1,k}=q\,x_{1,k}x_{1,j}$. If $n=1$, then for $i>\ell$, we set $x_{i,1}x_{\ell,1}=q\,x_{\ell,1}x_{i,1}$. 
\end{quantumalgebra}

\begin{qtrans}
In this work we always fix $q\in\mathbb{K}^*$ to be transcendental over $\mathbb{Q}$. 
\end{qtrans}

It is known that $\mathcal{A}$ can be presented as an iterated Ore extension and hence is a Noetherian domain. Furthermore, $\mathcal{A}$ is a domain of finite GK dimension. We may therefore conclude by Proposition 4.13 in \cite{kl} the existence of the quotient division algebra ${\rm Frac}(\mathcal{A})$. Finally, since we have assumed that $q$ is, in particular, not a root of unity, every prime ideal is completely prime (see \cite{gl3}).  Denote by spec($\mathcal{A})$ the set of prime ideals in $\mathcal{A}$. We equip spec($\mathcal{A})$ with the Zariski topology. 

Although our results are to be applied to $\mathcal{A}$, our calculations will mostly be performed in the following algebra $\mathcal{B}$.

\begin{Talgebra} \label{talg}
For positive integers $m$ and $n$, the $\mathbb{K}$-algebra $\mathcal{A}^\prime$ is the $m\times n$ quantum affine space. It is the algebra generated by $\langle t_{i,j}\mid (i,j)\in [m]\times[n]\rangle$ where commutation amongst the \emph{canonical generators} $t_{i,j}$ is defined as follows. Consider the $m\times n$ matrix $T_{\mathcal{A}^\prime}$ with $T_{\mathcal{A}^\prime}[i,j] = t_{i,j}$. If $m\geq 2$ and $n\geq 2$, then for any $2\times 2$ submatrix
$\left[\begin{array}{cc}
a & b \\
c & d
\end{array}\right]$
 of $T_{\mathcal{A}^\prime}$, the following hold:
\begin{enumerate}
	\item $ab= q\,ba$ and $cd= q\,dc$,
	\item $ac= q\,ca$ and $bd= q\,db$,
	\item $bc=cb$,
	\item $ad=da$.
\end{enumerate}
If $m=1$, then for $j>k$, we set $t_{1,j}t_{1,k}=q\,t_{1,k}t_{1,j}$. If $n=1$, then for $i>\ell$, we set $t_{i,1}t_{\ell,1}=q\,t_{\ell,1}t_{i,1}$. 

We denote by $\mathcal{B}$ the McConnell-Pettit algebra associated to the quantum affine space $\mathcal{A}^\prime$. It is the localization of $\mathcal{A}^\prime$ at the regular normal elements $t_{i,j}$ for $(i,j)\in [m]\times[n]$. Finally, we set $T_\mathcal{B}=T_{\mathcal{A}^\prime}$.
\end{Talgebra}

Note that by work of Cauchon (Th\'eor\`eme 2.2.1 of~\cite{cauchon2}) $\mathcal{A}$ has a localization isomorphic to $\mathcal{B}$.
 
Consider the following action of $\mathcal{H}=(\mathbb{K}^*)^{m+n}$ by $\mathbb{K}$-automorphisms on $\mathcal{A}$. The element $h=(\rho_1,\ldots,\rho_m,\gamma_1,\ldots,\gamma_n)\in(\mathbb{C^*})^{m+n}$ acts on a canonical generator  $x_{i,j}$ by, $$h\cdot x_{i.j} := \rho_i\gamma_jx_{i,j}.$$ 

An ideal $I$ is \textit{$\mathcal{H}$-invariant} if $h\cdot I = I$ for all $h\in \mathcal{H}$. The set of all $\mathcal{H}$-invariant prime ideals is denoted by $\mathcal{H}$-spec($\mathcal{A}$). The results of Goodearl and Letzter \cite{gl} imply the following.

\begin{goodearlletzter}[$\mathcal{H}$-stratification of spec($\mathcal{A}$)]
Let $\mathcal{A}$ be the algebra of $m\times n$ quantum matrices. The following hold.
\begin{enumerate}
\item $\mathcal{H}$-spec($\mathcal{A}$) is a finite set.
\item
The set {\rm spec}($\mathcal{A}$) can be partitioned (or \emph{stratified}) into a disjoint union as follows.

$$\emph{spec}(\mathcal{A}) = \bigcup_{J\in \mathcal{H}\text{\emph{-spec}}(\mathcal{A})} Y_J,$$

where $\displaystyle{Y_j:=\{P\in\text{\emph{spec}}(\mathcal{A}) \mid \bigcap_{h\in\mathcal{H}} h\cdot P = J\}}$.
\item Each stratum $Y_J$ is homeomorphic to the prime spectrum of a commutative Laurent polynomial ring over $\mathbb{K}$.
\end{enumerate}
\end{goodearlletzter}

\subsection{Cauchon Diagrams}

We begin this section by briefly describing Cauchon's deleting derivations algorithm \cite{cauchon2} applied to $\mathcal{A}$. The algorithm constructs, for each $(s,t)\in ([m]\times [n])\cup (m,n+1)$, a matrix $X^{(s,t)}$ with entries in the quotient division ring ${\rm Frac}(\mathcal{A})$ as follows. 

First set $X^{(m,n+1)}:=X_\mathcal{A}$. Now assume that $X^{(s,t)^+}$ has been constructed. If $(i,j) \not\leq (s-1,t-1)$, then $X^{(s,t)}[i,j]=X^{(s,t)^+}[i,j]$. For each $(i,j)\leq(s-1,t-1)$, we consider the $2\times 2$ submatrix $$X^{(s,t)^+}[\{i,s\},\{j,t\}] = 
\left[\begin{array}{cc}
X^{(s,t)^+}[i,j]	& X^{(s,t)^+}[i,t] \\
X^{(s,t)^+}[s,j] &  X^{(s,t)^+}[s,t]
\end{array}\right] :=
\left[\begin{array}{cc}
x	& y \\
z & w
\end{array}\right].$$
From this we set $$X^{(s,t)}[i,j] := x - yw^{-1}z\in {\rm Frac}(\mathcal{A}).$$

Let $\mathcal{A}^{(s,t)}\subseteq {\rm Frac}(\mathcal{A})$ denote the algebra whose matrix of canonical generators is $X^{(s,t)}$. Cauchon \cite{cauchon1} has shown that there exists an embedding $\text{spec}(\mathcal{A}^{(s,t)^+}) \hookrightarrow \text{spec}(\mathcal{A}^{(s,t)})$. Furthermore, the final algebra $\mathcal{A}^{(1,1)}$ is isomorphic to the quantum affine space $\mathcal{A}^\prime$. The composition of all such embeddings gives an embedding $\psi:\text{spec}(\mathcal{A})\hookrightarrow \text{spec}(\mathcal{A}^\prime)$. Finally, Cauchon has shown that the image of $\mathcal{H}$-spec($\mathcal{A}$) under $\psi$ is parametrized by a useful set of combinatorial objects, which we call \emph{Cauchon diagrams}.

\begin{cauchondiagram} \label{cauchondiagram}
Let $\mathcal{C}$ be an $m\times n$ grid of squares where we have coloured each square either white or black. Call $\mathcal{C}$ a \emph{Cauchon diagram} if, for every black square, either every square above it or every square to its left is also black.
We let $W_\mathcal{C}$ be the set of white squares and $B_\mathcal{C}$ to be the set of black squares. We index the squares of an $m\times n$ Cauchon diagram as one would a matrix. That is, the square in the $i$th row from the top and $j$th column from the left is called the $(i,j)$ square. 

 \end{cauchondiagram}
We have already seen an example of a Cauchon diagram in Figure~\ref{example}. Cauchon~\cite{cauchon2} proved that for every Cauchon diagram $\mathcal{C}$ there is a unique $\mathcal{H}$-prime $J_\mathcal{C}$ of $\mathcal{A}$ such that $\psi(J_\mathcal{C})=\left\langle t_{i,j}\mid (i,j)\in B_\mathcal{C}\right\rangle $ and that $$\mathcal{H}{\rm -spec}(\mathcal{A})=\{J_\mathcal{C}\mid \textnormal{ $\mathcal{C}$ an $m\times n$ Cauchon diagram}\}.$$

\section{Cauchon Graphs}
We essentially follow Postnikov \cite{postnikov} by defining a weighted directed graph given a Cauchon diagram $\mathcal{C}$. If $(i,j)$ is a white square in $\mathcal{C}$, then let $(i,j^-)$ be the first white square to its left (if one exists) and $(i,j^+)$ the first white square to its right (if one exists). Similarly, let $(i^+,j)$ be the first white square below $(i,j)$ (if one exists). 

\begin{cauchongraph} \label{cauchongraph}
Let $\mathcal{C}$ be an $m\times n$ Cauchon diagram.
The \textit{Cauchon graph} $\mathcal{G}_\mathcal{C}=(V,\vec{E},w)$ is an edge-weighted directed graph defined as follows. The vertices consist of the set $V=W_\mathcal{C}\cup\{r_1,\ldots,r_m\}\cup\{c_1,\ldots,c_n\}:=W_\mathcal{C}\cup R\cup C$. The set $\vec{E}$ of directed edges and a weight function $w:\vec{E}(\mathcal{G}_\mathcal{C})\rightarrow \mathcal{B}$ are constructed as follows. 
\begin{enumerate}
	\item For every $i\in [m]$, put a directed edge from $r_i$ to the rightmost white square in row $i$ (if it exists), say $(i,k)$. Give these edges weight $t_{i,k}$. 
	\item For every column $j\in [n]$, put a directed edge from the bottom-most white square in column $j$ (if it exists) to the vertex $c_j$. Give these edges weight 1.
	\item For every $(i,j)\in W_\mathcal{C}$, put a directed edge from $(i,j)$ to $(i,j^-)$ (if it exists). Give these edges a weight $t^{-1}_{i,j}t_{i,j^-}$.
	\item For every $(i,j)\in W_\mathcal{C}$, put a directed edge from $(i,j)$ to $(i^+,j)$ (if it exists). Give each of these edges a weight of 1.
\end{enumerate}
\end{cauchongraph}
For convenience, we always assume a Cauchon graph is embedded in the plane in the following way. First, place a vertex in each white square of $\mathcal{C}$ and label this vertex by the coordinates of the white square. Next, place a vertex to the right of each row and below each column. The vertex to the right of row $i$ is labelled $r_i$, and the vertex below column $j$ is $c_j$. See Figure~\ref{examplecg}. Under this embedding we may unambiguously use directional terms such as horizontal, vertical, above, below, left and right when discussing a Cauchon graph.

\begin{figure}
\centering
\includegraphics[height=6cm]{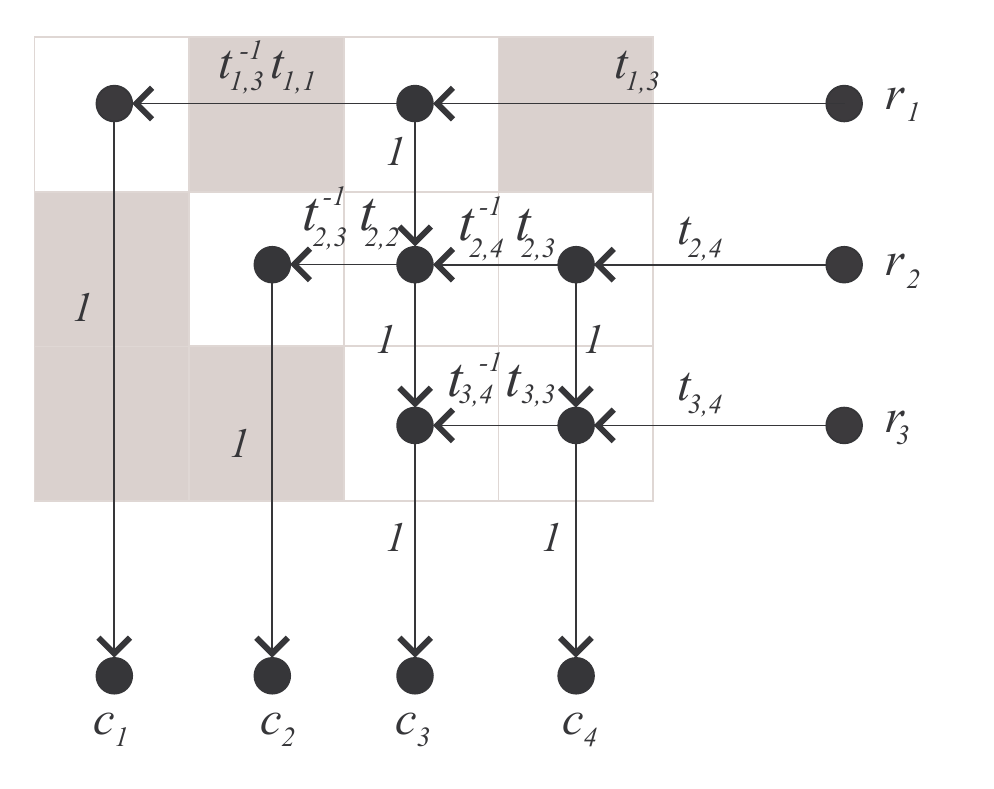}
\caption{A Cauchon graph superimposed on top of its Cauchon diagram.}\label{examplecg}
\end{figure}

\begin{graphnotation}
Consider the Cauchon graph for some fixed $m\times n$ Cauchon diagram.
\begin{itemize}
	\item For $I\subseteq [m]$ and $J\subseteq [n]$, we set $R_I=\{r_i\in R\mid i\in I\}$ and $C_J=\{c_j\in C\mid j\in J\}$.
	\item Let $e=\left((i,k),(i,j)\right)$ be a horizontal edge with both endpoints in $W_\mathcal{C}$ (so $k>j$). We set $\text{row}(e)= i$, $\text{col}_1(e)=k$ and $\text{col}_2(e)=j$. The pair $\{k,j\}$ will be denoted by $\text{col}(e)$. In other words, $\text{col}_1(e)$ is the column containing the right end (or tail) of $e$ and $\text{col}_2(e)$ is the column containing the left end (or head) of $e$.
	%%% do we need? P...
	\item All paths in this paper are assumed to be \textit{directed} paths. When we wish to emphasize the first vertex $v_0$ and final vertex $v_n$ of a directed path $P$ we will write $P:v_0\Rightarrow v_n$.
	\item The weight of a path $P=(v_0,e_1,v_1,e_2,\ldots,e_n,v_n)$ is the product of the weights of its edges, multiplied from left to right in the order they appear in $P$. In other words, $$w(P)=w(e_1)w(e_2)\cdots w(e_n) = \prod_{i=1}^n w(e_i).$$ 
	\item If $K:v_0\Rightarrow v$ and $L:v\Rightarrow v_n$ are two paths in a Cauchon graph, then  $KL=KL: v_0\Rightarrow v_n$ is the path obtained by appending $L$ to $K$ in the obvious way. That is, $KL$ is the path which travels along $K$ from $v_0$ to $v$, and then continues on $L$ from $v$ to $v_n$. Note that $KL$ is still a directed path since, by the next proposition, Cauchon graphs are acyclic.
\end{itemize}
\end{graphnotation}

\begin{cauchongraphproperties} \label{cgproperties}
For a Cauchon diagram $\mathcal{C}$, the Cauchon graph $\mathcal{G}_\mathcal{C}$ has the following properties:
\begin{enumerate}
\item The graph $\mathcal{G}_\mathcal{C}$ is acyclic, i.e., it has no \textit{directed} cycles.
\item The embedding described above is a planar embedding, i.e., no edges touch except at a vertex.
\item If the path $P:(i,j_2)\Rightarrow(i,j_1)$ consists only of horizontal edges, then $$w(P)=t_{i,j_2}^{-1}t_{i,j_1}.$$
\end{enumerate}
\end{cauchongraphproperties}
\begin{proof}
As all edges are directed either right to left or top to bottom, the first property is obvious. To see planarity, if two edges cross, then these edges must consist of one vertical and one horizontal edge, and their intersection point corresponds to a black square. This implies that we have a black square in $C$ with both a white square above and a white square to its left, contradicting the definition of a Cauchon diagram.

Finally, if $(i,k)$ is an internal vertex of $P:(i,j_1)\Rightarrow(i,j_2)$ (i.e., $k\neq j_1,j_2$), then $(i,k^-)$ and $(i,k^+)$ exist. Now the edge $e_1:=((i,k^+),(i,k))\in P$ has weight $t_{i,k^+}^{-1}t_{i,k}$ and the edge $e_2= ((i,k),(i,k^-))\in P$ has weight $t_{i,k}^{-1}t_{i,k^-}$. Therefore, $w(e_1)w(e_2)  = t_{i,k^+}^{-1}t_{i,k^-}$.  It follows that $w(P)$ is a telescoping product and so clearly $w(P)=t_{i,j_2}^{-1}t_{i,j_1}$.
\end{proof}

Since the edge weights are in a noncommutative ring, the remainder of this section is devoted to a sequence of lemmas which give commutation relations between edges and between certain paths in a Cauchon graph. These lemmas will be needed only in the proof of Theorem~\ref{gv}.

\begin{edgecommutivity} \label{edgecommutivity}
Let $\mathcal{C}$ be a Cauchon diagram. Let $e$ and $f$ be horizontal edges in $\mathcal{G}_\mathcal{C}$ with both endpoints in $W_\mathcal{C}$ and such that $\emph{row}(f)\leq \emph{row}(e)$.
\begin{enumerate}
\item If $\emph{col}(e)\cap \emph{col}(f)=\emptyset$, then $w(f)w(e)=w(e)w(f)$\label{ec1}.
\item If $|\emph{col}(e)\cap \emph{col}(f)|=1$, then: 
	\begin{description}\label{ec2}
	\item[$i.$] $w(f)w(e) = q\,w(e)w(f)$, if $\emph{col}_i(e)=\emph{col}_i(f)$ for $i=1$ or $i=2$, and
	\item[$ii.$]$w(f)w(e) = q^{-1}\,w(e)w(f)$ otherwise.
	\end{description}
\item If $|\emph{col}(e)\cap \emph{col}(f)|=2$, then $w(f)w(e)=q^2 w(e)w(f)$.\label{ec3}
\end{enumerate}
\end{edgecommutivity}

\begin{figure}
\centering
\includegraphics[height=12cm]{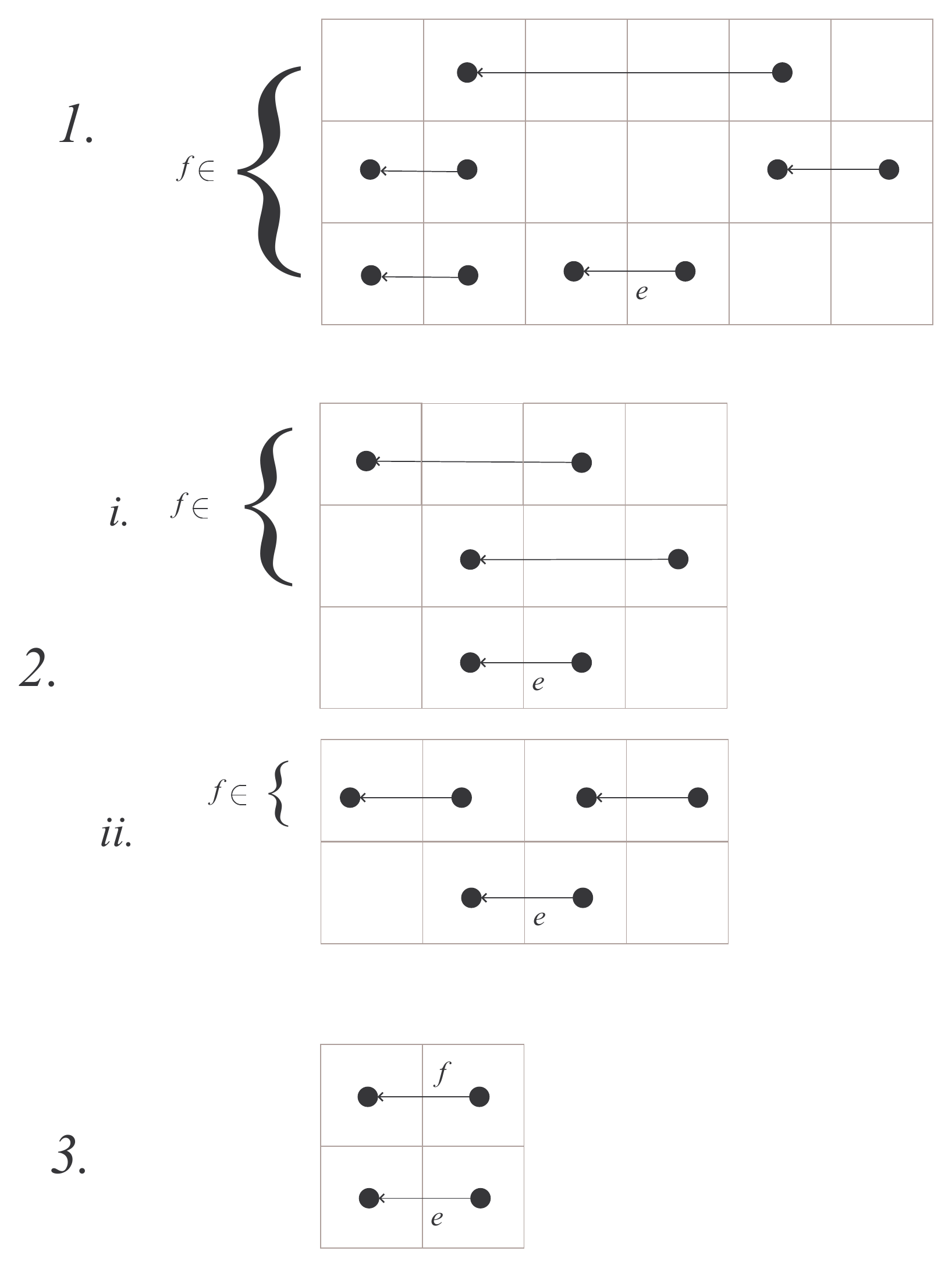}
\caption{Examples of the different cases in Lemma~\ref{edgecommutivity}.} \label{ecomm}
\end{figure}

\begin{proof}
We verify the case $\text{col}(e)\cap \text{col}(f)=\emptyset$. The other cases can be disposed of in a similar manner by checking the possibilities (see Figure~\ref{ecomm}). 
So suppose that $\text{col}(e)\cap \text{col}(f)=\emptyset$.
 
First note that if $j$ is such that $\text{col}_2(e)<j<\text{col}_1(e)$, then it follows that entry $(\text{row}(e),j)$ is black in $\mathcal{C}$, and since $(\text{row}(e),\text{col}_2(e))$ is a white square to its left, we must have that $(i,j)$ is black for every $i\leq\text{row}(e)$. In other words, no horizontal edge in $\mathcal{G}_\mathcal{C}$ has an endpoint whose column coordinate lies strictly between the column coordinates of $e$. 

Now if $\text{row}(e)\neq \text{row}(f)$, then $w(e)$ and $w(f)$ clearly commute by the definition of the algebra $\mathcal{B}$. Suppose that $\text{row}(e)= \text{row}(f)$. Say $w(e)=t_{i,j_1}^{-1}t_{i,j_2}$ and $w(f)=t_{i,j_3}^{-1}t_{i,j_4}$ where $j_1>j_2>j_3>j_4$. We have 
\begin{eqnarray*}
w(e)w(f) & = & (t_{i,j_1}^{-1}t_{i,j_2})(t_{i,j_3}^{-1}t_{i,j_4})\\ 
& =&  (q^{-1} q)\  t_{i,j_3}^{-1}(t_{i,j_1}^{-1}t_{i,j_2})t_{i,j_4} \\
& =& (q^{-1}q)(qq^{-1})\ (t_{i,j_3}^{-1}t_{i,j_4})(t_{i,j_1}^{-1}t_{i,j_2})\\
&=& w(f)w(e).
\end{eqnarray*}
\end{proof}

Note that Lemma~\ref{edgecommutivity}, parts (1) and (2) remain true if $e$ or $f$ is an edge which has an endpoint in $R$. 

\begin{splitpath}\label{splitpath}
Let $K:v_0\Rightarrow v$ and $L:v\Rightarrow v_t$ be directed paths in a Cauchon graph.
\begin{enumerate}
\item If either $K$ or $L$ contain only vertical edges, then $w(K)w(L)=w(L)w(K)$.
\item If both $K$ and $L$ contain a horizontal edge, then $w(K)w(L)=q^{-1}\,w(L)w(K)$.
\end{enumerate}
\end{splitpath}

\begin{figure}
\centering
\includegraphics[height=6cm]{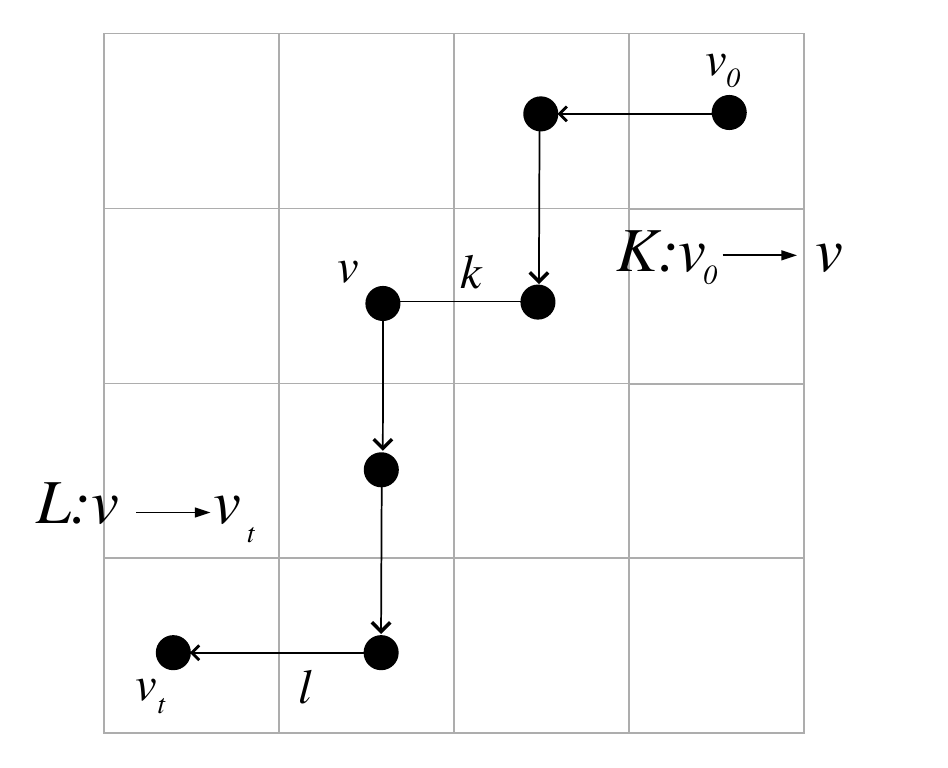}
\caption{Example of a situation in Lemma~\ref{splitpath}.}
\end{figure}

\begin{proof}
We need only consider the horizontal edges of $K$ and $L$ since all vertical edges have weight $1$ and so commute with everything. Now if either $K$ or $L$ contain only vertical edges, then $w(K)=1$ or $w(L)=1$ and so $w(K)$ and $w(L)$ commute. 

Suppose that both $K$ and $L$ contain horizontal edges. Let $k$ be the last horizontal edge in $K$ and let $l$ be the first horizontal edge in $L$. By the embedding of a Cauchon graph, the horizontal edges of $L$ are always to the left of horizontal edges of $K$ or ``south-west'' of $K$. When computing $w(K)w(L)$, the only edge weights which do not commute are $w(k)$ and $w(\ell)$ by Lemma~\ref{edgecommutivity}\,(1). By Lemma~\ref{edgecommutivity}\, $(2ii)$ we obtain
\begin{eqnarray*}
w(K)w(L) & = & w(K\setminus\{k\})w(k)w(l)w(L\setminus\{l\})\\
& = & q^{-1}\,w(K\setminus\{k\})w(l)w(k)w(L\setminus\{l\})\\
 &= & q^{-1}\,w(l)w(L\setminus\{l\})w(K\setminus\{k\})w(k)\\
 & = & q^{-1}\,w(L)w(K)
\end{eqnarray*}
\end{proof}

\begin{disjointpath} \label{disjointpath}
Let $\mathcal{G}_\mathcal{C}$ be a Cauchon graph and let $K:v\Rightarrow c_i$ and $L:v\Rightarrow c_j$ be two directed paths with only their initial vertex in common. Let $K$ be the path that starts with a horizontal edge and $L$ be the path that starts with a vertical edge. The weights of the two paths commute as follows.
\begin{enumerate}
\item If $L$ consists only of vertical edges (or no edges at all), then $w(K)w(L) = w(L)w(K)$.
\item If $L$ has a horizontal edge then $w(K)w(L) = q\,w(L)w(K)$.
\end{enumerate}
\end{disjointpath}
\begin{proof}
Case 1 is obvious since here we have $w(L)=1$, so we may suppose that $L$ has at least one horizontal edge. By Lemma~\ref{edgecommutivity}, any horizontal edge $e$ in $L$ commutes with any edge in $K$ except those whose column coordinates intersect the set $\text{col}(e)$. Recall from the proof of Lemma~\ref{edgecommutivity} that no edge in $K$ has an endpoint in between (with respect to column coordinates) the endpoints of edges in $L$.

If $f_1$ is the first horizontal edge in $K$ and $e_1$ is the first horizontal edge in $L$, then we have $\text{r.col}(e_1)=\text{r.col}(f_1)$. There are two cases to consider (see Figure~\ref{lem3}).

\begin{description}
\item[Case $(i)$:] $\text{l.col}(f_1)<\text{l.col}(e_1)$. Here we have that, by Lemma~\ref{edgecommutivity} $(2i)$, $$w(f_1)w(e_1)=q\, w(e_1)w(f_1).$$
\item[Case $(ii)$:] $\text{l.col}(f_1)=\text{l.col}(e_1)$. In this case, the second horizontal edge $f_2$ of $K$ satisfies $\text{r.col}(f_2)=\text{l.col}(e_1)$ and $\text{l.col}(f_2)<\text{l.col}(e_1)$. Applying Lemma~\ref{edgecommutivity} twice we find 
\begin{eqnarray*}
w(f_1)w(f_2)w(e_1) & = & w(f_1) q^{-1}w(e_1)w(f_2) \text{, by Lemma~\ref{edgecommutivity} (2.i),}\\
& = & (q^{-1}q^2)\,w(e_1)w(f_1)w(f_2) \text{, by Lemma~\ref{edgecommutivity} (3),}\\
& = & q\,w(e_1)w(f_1)w(f_2).
\end{eqnarray*}
\end{description}
It follows that $w(K)w(e_1)=q\,w(e_1)w(K)$.

\begin{figure} 
\centering
\includegraphics[height=8cm]{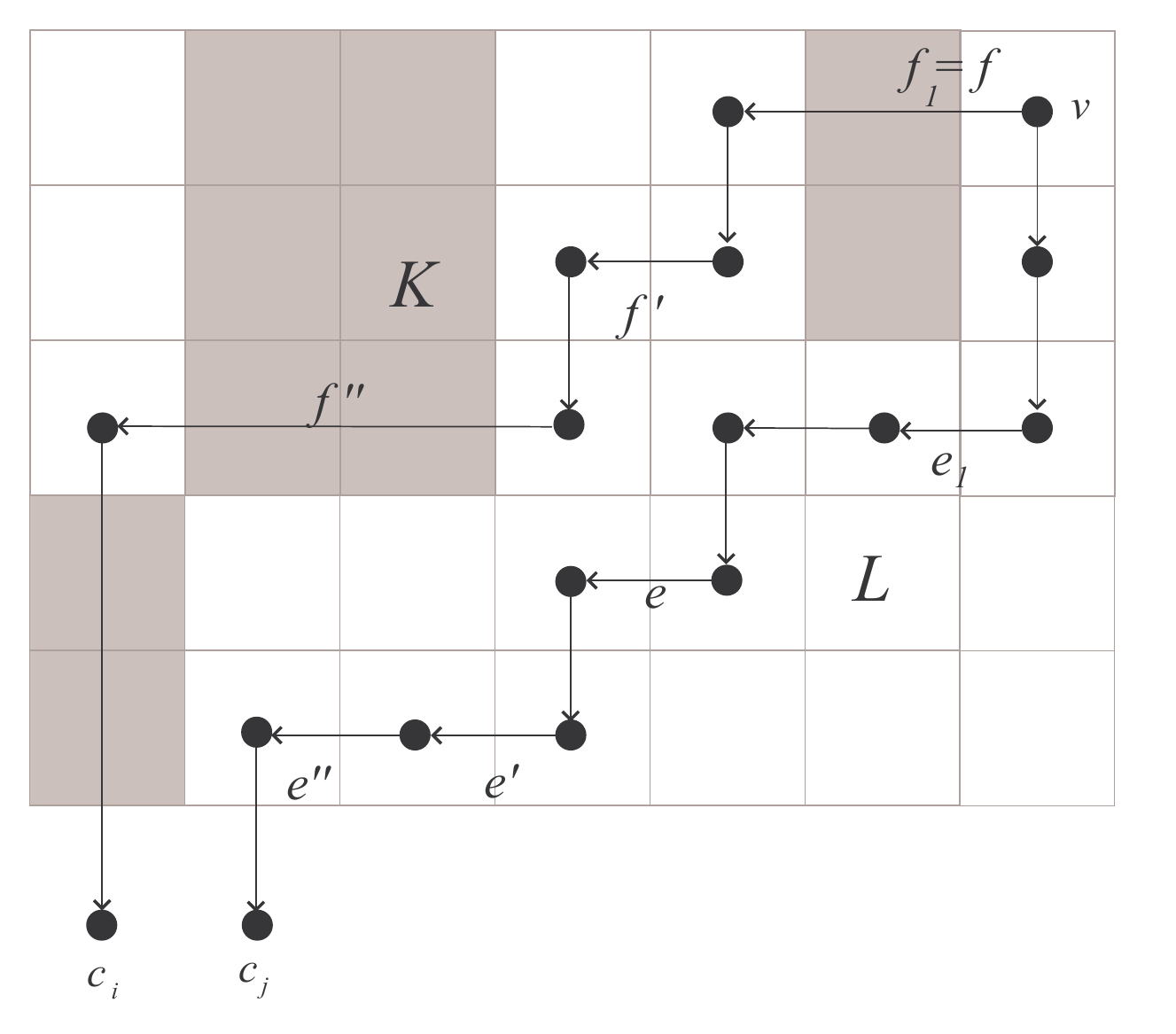}
\caption{Example of a typical situation of Lemma~\ref{disjointpath}.}\label{lem3}
\end{figure}

Now suppose that $e$ is \emph{not} the first horizontal edge in $L$. We show that $w(e)w(K)=w(K)w(e)$. There are exactly three possibilities for the edge $e$.

\begin{description}
\item[Case (a):] Every edge $f$ in $K$ satisfies $\text{col}(f)\cap\text{col}(e) = \emptyset$. For an example of an edge which falls in this case, see edge $e^{\prime\prime}$ in Figure~\ref{lem3}. By Lemma~\ref{edgecommutivity}\,(1), it follows that $w(e)w(K)=w(K)w(e)$. 

\item[Case (b):] There exist two distinct edges $f^\prime$ and $f^{\prime\prime}$ in $K$ such that $|\text{col}(f^\prime)\cap\text{col}(e)|=1$ and $|\text{col}(f^{\prime\prime})\cap\text{col}(e)|=1$, and $\text{col}(g)\cap\text{col}(e) = \emptyset$ for all other edges $g$ in $K$. For an example of an edge which falls in this case, see edge $e^\prime$ in Figure~\ref{lem3}. Now together with $e$, both $f^\prime$ and $f^{\prime\prime}$ fall under case (2) of Lemma~\ref{edgecommutivity}. Furthermore, exactly one is in subcase $(2i)$ of that lemma, while the other is in subcase $(2ii)$. Without loss of generality, we suppose that $f^\prime$ is in case $(2ii)$ of Lemma~\ref{edgecommutivity}. We have
\begin{eqnarray*}
w(f^{\prime\prime})w(f^\prime)w(e) & = & w(f^{\prime\prime})q^{-1}\,w(e)w(f^\prime) \\
& = & q\,w(e)w(f^{\prime\prime})q^{-1}\,w(f^\prime)\\
& = & w(e)w(f^{\prime\prime})w(f^\prime).
\end{eqnarray*}

Since by Lemma~\ref{edgecommutivity}\,(1) we have that $w(e)$ commutes with $w(g)$ for every edge $g$ in $K$ with $\text{col}(g)\cap\text{col}(e) = \emptyset$, it follows that we again have $w(e)w(K)=w(K)w(e)$, as desired.

\item[Case (c):] There exist three distinct edges $f$,$f^\prime$ and $f^{\prime\prime}$ in $K$ such that $e$ and $f^\prime$ fall under Lemma~\ref{edgecommutivity} (3), while $e$, together with either $f$ or $f^{\prime\prime}$ fall into Lemma~\ref{edgecommutivity} $(2ii)$. Furthermore, $\text{col}(g)\cap\text{col}(e) = \emptyset$ for all other edges $g$ in $K$. For an example of this case, see edge $e$ in Figure~\ref{lem3}. We have

\begin{eqnarray*}
w(f^{\prime\prime})w(f^\prime)w(f)w(e) & = &  w(f^{\prime\prime})w(f^\prime)q^{-1}w(e)w(f) \\
& = & w(f^{\prime\prime})q^2\, w(e)w(f^\prime)q^{-1}w(f)\\
& = & q^{-1}\,w(e)w(f^{\prime\prime})q^2\,w(f^\prime)q^{-1}w(f)\\
& = & w(e)w(f^{\prime\prime})w(f^\prime)w(f).
\end{eqnarray*}
\end{description}

Since by Lemma~\ref{edgecommutivity}\,(1) we know $w(e)$ commutes with $w(g)$ for every edge $g$ in $K$ with $\text{col}(g)\cap\text{col}(e) = \emptyset$, it follows that we again have $w(e)w(K)=w(K)w(e)$, as desired. This completes the analysis of all possibilities for $e$ not being the first horizontal edge in $L$, and so we conclude that $w(e)w(K)=w(K)w(e)$ for all such edges $e$.  Hence
\begin{eqnarray*}
w(K)w(L) & = & w(K)w(e_1)w(L\setminus e_1)\\
 & = & q\,w(e_1)w(K)w(L\setminus e_1)\\
 & = & q\,w(e_1)w(L\setminus e_1)w(K)\\
 & = & q\,w(L)w(K).
 \end{eqnarray*}
\end{proof}

\section{vertex-disjoint Path Systems and $q$-Determinants}

In this section we give one of the main tools used in the proof of Theorem~\ref{main}. We begin with some definitions.

\begin{pathsystem}
Let $\mathcal{C}$ be an $m\times n$ Cauchon diagram. Let $I=\{i_1,\ldots,i_k\}\subseteq [m]$ and $J=\{j_1,\ldots,j_k\}\subseteq [n]$ be two subsets of equal size with $i_1<i_2<\cdots<i_k$ and $j_1<j_2<\cdots j_k$.

An \textit{$(R_I,C_J$)-path system} $\mathcal{P}$ is a set of $k$ directed paths in $\mathcal{G}_\mathcal{C}$, each starting at different a vertex in $R_I$ and each ending at a different vertex in $C_J$. Note the following.

\begin{itemize} 
\item There exists a permutation $\sigma_\mathcal{P}\in S_k$ such that 
$$\mathcal{P} = \{P_\ell:r_{i_\ell}\Rightarrow c_{j_{\sigma_\mathcal{P}(\ell)}}\mid \ell\in [k]\}.$$

\item The \emph{$q$-sign} of $\mathcal{P}$ is the quantity $$\text{sgn}_q(\mathcal{P})=(-q)^{\ell(\sigma_\mathcal{P})},$$ where we recall that $\ell(\sigma_\mathcal{P})$ is the length of the permutation $\sigma_\mathcal{P}$ as defined in Notation~\ref{basicnotation}.

\item A path system is \textit{vertex-disjoint} if no two paths share a vertex.
\item The weight of $\mathcal{P}$ is the product $\prod_{\ell=1}^k w(P_\ell) = w(P_1)w(P_2)\cdots w(P_k)$.
\end{itemize}
\end{pathsystem}

We will need the following easy lemma.

\begin{consecutivepaths} \label{consec}
Let $\mathcal{C}$ be a Cauchon diagram and let $I\subseteq [m]$ and $J\subseteq[n]$ be two sets of cardinality $k$. If $\mathcal{P}=\{P_1,\ldots,P_k\}$ is a \emph{non}-vertex-disjoint $(R_I,C_J)$-path system in $\mathcal{G}_\mathcal{C}$, then there exists an $i$ such that $P_i$ and $P_{i+1}$ share a vertex.
\end{consecutivepaths}
\begin{proof}
Let $$d=\min \{|i-j| \mid i\neq j \text{ and $P_i$ and $P_j$ share a vertex\}}.$$

Observe that $d$ is well-defined and at least 1 since, by assumption, there exists at least one pair of intersecting paths in $\mathcal{P}$. Let $P_i$ and $P_j$ be two paths which achieve this minimum. Hence there is a first vertex $x\in W_\mathcal{C}$ at which they intersect. If $r_i$ and $r_j$ are the first vertices on the paths $P_i$ and $P_j$ respectively, then the two subpaths $P_i^\prime:r_i\Rightarrow x$ and $P_j^\prime:r_j\Rightarrow x$ together with a new vertical edge $(r_i,r_j)$ form a closed simple loop $L$ in the plane. See Figure~\ref{fig6}.

\begin{figure}
\centering
\includegraphics[height=6cm]{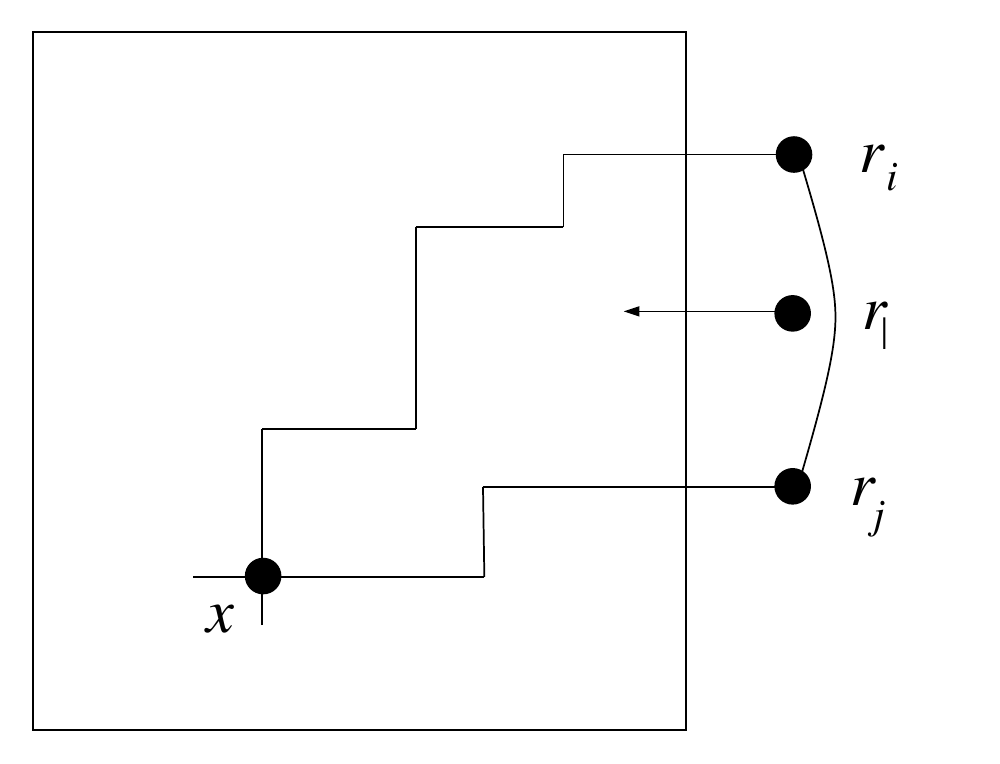}
\caption{Visualization of the proof of Lemma~\ref{consec}. Here we see that the path starting at $r_\ell$ must intersect either $r_i$ or $r_j$ in order to reach $c_{\sigma_\mathcal{P}(\ell)}$}\label{fig6}
\end{figure}

If $d>1$, then there exists an $\ell\in I$ such that $r_\ell$ lies between $r_i$ and $r_j$ in $\mathcal{G}_\mathcal{C}$. But in order for $P_\ell$ to reach its endpoint in $C$, we must have that an internal vertex of $P_\ell$ intersects $L$. This intersection point occurs at a vertex by planarity. Hence $P_\ell$ shares a vertex with either $P_i$ or $P_j$, which contradicts the minimality of $d$.

\end{proof}

\begin{pathmatrixdef}
Let $\mathcal{C}$ be an $m\times n$ Cauchon diagram. The \textit{path matrix} of $\mathcal{G}_\mathcal{C}$ is the $m\times n$ matrix $M_\mathcal{C}$ with $$M_\mathcal{C}[i,j]=\sum_{P:r_i\Rightarrow c_j}w(P),$$ where the sum is over all possible directed paths in $\mathcal{G}_\mathcal{C}$ starting at $r_i$ and ending at $c_j$. If no such path exists, then $M_\mathcal{C}[i,j]=0$.
\end{pathmatrixdef}

Now we are ready to prove the main theorem of this section, whose statement and proof are very much in the spirit of Lindstr\"{o}m's Lemma. 
\begin{gessel}\label{gv}
Let $\mathcal{C}$ be an $m\times n$ Cauchon diagram. If $I\subseteq [m]$ and $J\subseteq [n]$ are two sets of size $k$, then $${\det}_q(M_\mathcal{C}[I,J])=\sum_{\mathcal{P}}w(\mathcal{P}),$$
where the sum is over all \textbf{vertex-disjoint} $(R_I,C_J)$-path systems in $\mathcal{G}_\mathcal{C}$.
\end{gessel}
\begin{proof}

In order to simplify the presentation of this proof we will take $I=J=\{1,\ldots,k\}$. The proof of the general case is essentially the same but notationally more cumbersome. 

To begin, note that 
\begin{eqnarray*}
{\det}_q(M_\mathcal{C}[I,J]) & = & \sum_{\sigma\in S_k} \text{sgn}_q(\sigma)\left(\prod_{i=1}^k M_\mathcal{C}[i,\sigma(i)]\right) \\
& = & \sum_{\sigma} \text{sgn}_q(\sigma)\left(\prod_{i=1}^k\left(\sum_{P:r_i\Rightarrow c_{\sigma(i)}} w(P)\right)\right) \\
& = & \sum_{\text{$(R_I,C_J$)-path systems $\mathcal{P}$}}\text{sgn}_q(\mathcal{P})w(\mathcal{P}).
\end{eqnarray*}

Let $\mathcal{N}$ be the set of \textit{non}-vertex-disjoint ($R_I,C_J$)-path systems. We claim that $$\sum_{\mathcal{P}\in\mathcal{N}}\text{sgn}_q(\mathcal{P})w(\mathcal{P}) = 0.$$ To show this, we find a fixed-point free involution $\pi:\mathcal{N}\rightarrow\mathcal{N}$ with the property that for every $\mathcal{P}\in\mathcal{N}$, 
\begin{eqnarray}\label{eqn1}
\text{sgn}_q(\mathcal{P})w(\mathcal{P})& = &-\text{sgn}_q(\pi(\mathcal{P}))w(\pi(\mathcal{P})),
\end{eqnarray}
where $\pi(\mathcal{P}):=\{\pi(P_1),\ldots,\pi(P_k)\}.$

Suppose that $\mathcal{P}=\{P_1,\ldots,P_k\}\in\mathcal{N}$. Define $\pi$ as follows. Let $i$ be the minimum index of $I$ such that $P_{i}$ and $P_{i+1}$ intersect (which exists by Lemma~\ref{consec}). Let $x$ be the \textit{last} vertex which they have in common. Let $K_1:r_i \Rightarrow x$ and $L_1:x\Rightarrow c_{\sigma_\mathcal{P}(i)}$ be the two subpaths of $P_i$ such that $P_i=K_1L_1$. Define $K_2$ and $L_2$ from $P_{i+1}$ similarly. Now we set

\begin{figure}
\centering
\includegraphics[height=6cm]{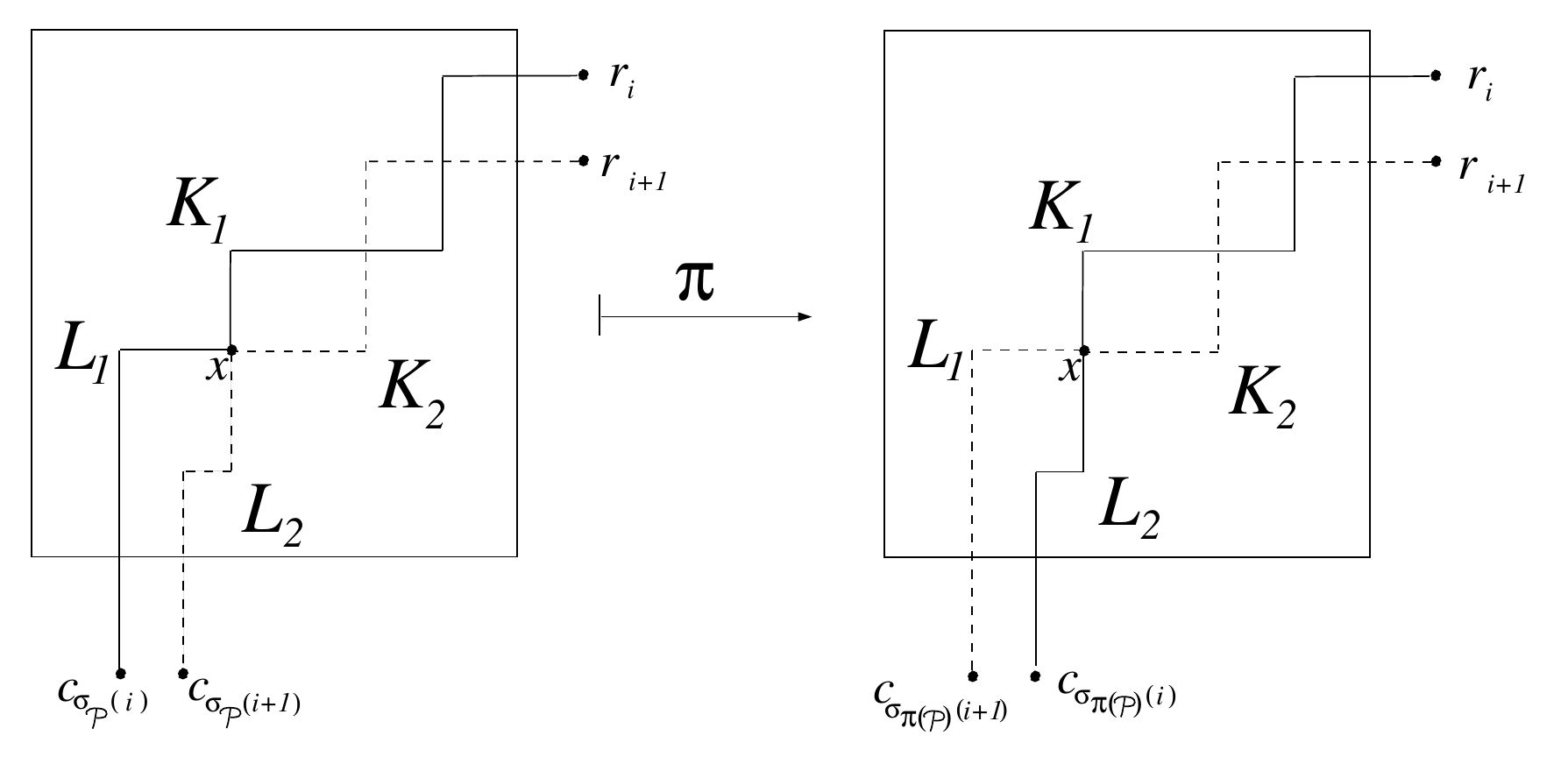}
\caption{Example of how $\pi$ acts on two intersecting paths. On the left hand side, $P_i$ is the solid path and $P_{i+1}$ is the dotted path. On the right hand side, $\pi(P_i)$ is the solid path, $\pi(P_{i+1})$ is the dotted path.}
\end{figure}

\begin{displaymath}
\pi(P_\ell) = \left\{ \begin{array}{ll}
K_1L_2 & \textrm{if $\ell=i$,}\\
K_2L_1 & \textrm{if $\ell={i+1}$,}\\
P_\ell & \textrm{otherwise.}
\end{array} \right.
\end{displaymath} 

It is clear from the definition that $\pi$ is an involution without fixed points so it remains to prove Equation~\ref{eqn1}. 
Since $\pi$ is an involution, we may assume, without loss of generality, that $\sigma_\mathcal{P}(i)<\sigma_\mathcal{P}(i+1)$. Thus $\sigma_{\pi(\mathcal{P})}=\sigma_\mathcal{P}(i\ \  i+1)$ and so $\sigma_{\pi(\mathcal{P})}$ has exactly one more inversion, i.e., 

\begin{eqnarray}\label{two}
\ell(\sigma_{\pi(\mathcal{P})}) & = & \ell(\sigma_\mathcal{P})+1.
\end{eqnarray}

Now consider $w(P_i)w(P_{i+1})$. There are two cases to consider. First suppose that $L_2$ has a horizontal edge. In this case, we find that 
\begin{eqnarray}
w(P_i)w(P_{i+1}) & = & w(K_1)w(L_1)w(K_2)w(L_2) \nonumber \\
& = & w(K_1)\,q\,w(K_2)w(L_1)w(L_2)  \text{  (Lemma~\ref{splitpath})}\nonumber \\
& = & w(K_1)\,q\,q\ w(K_2)w(L_2)w(L_1)  \text{  (Lemma~\ref{disjointpath})}\nonumber \\
& = & w(K_1)\,q\,q\,q^{-1}\,w(L_2)w(K_2)w(L_1)  \text{  (Lemma~\ref{splitpath})} \nonumber \\
& = & q\,w(\pi(P_i))w(\pi(P_{i+1})). \label{three}
\end{eqnarray}
If $L_2$ has only vertical edges, then a similar calculation shows again that $w(P_i)w(P_{i+1})= q\,w(\pi(P_i))w(\pi(P_{i+1}))$. Therefore,
\begin{eqnarray*}
w(\mathcal{P})& = & \left(\prod_{j=1}^{i-1} w(P_j)\right)w(P_i)w(P_{i+1})\left(\prod_{j=i+2}^k w(P_j)\right)\\
& = & \left(\prod_{j=1}^{i-1} w(\pi(P_j))\right)q\,w(\pi(P_i))w(\pi(P_{i+1}))\left(\prod_{j=i+2}^k w(\pi(P_j))\right)\\
& = & q\,w(\pi(\mathcal{P})).
\end{eqnarray*}

Finally, by Equations~\ref{two} and~\ref{three}, we obtain 
\begin{eqnarray*}
\text{sgn}_q(\mathcal{P})w(\mathcal{P}) + \text{sgn}_q(\pi(\mathcal{P}))w(\pi(\mathcal{P})) & = & \\
(-q)^{\ell(\sigma_\mathcal{P})}\,q\,w(\pi(\mathcal{P})) + (-q)^{\ell(\sigma_\mathcal{P})+1}w(\pi(\mathcal{P})) & = & 0.
\end{eqnarray*} This proves Equation~\ref{eqn1} and shows that $${\det}_q(M_\mathcal{C}[I,J])=\sum_{\mathcal{P}}\text{sgn}_q(\mathcal{P})w(\mathcal{P}),$$ where the sum is over all vertex-disjoint $(R_I,C_J)$ path systems. 

By Proposition~\ref{cgproperties},  $\mathcal{G}_\mathcal{C}$ is planar and so $\mathcal{P}$ cannot have any edge crossings. This implies that $\mathcal{P}=\{P_\ell:\ell\Rightarrow \ell\mid \ell=1,\ldots, k\}$. Thus $\sigma_\mathcal{P}$ is the identity permutation and so $\text{sgn}_q(\mathcal{P})=1$. Therefore, we obtain the desired equation in the statement of the theorem, namely, $${\det}_q(M_\mathcal{C}[I,J])=\sum_{\mathcal{P}}w(\mathcal{P}),$$ where the sum is over all vertex-disjoint $(R_I,C_J)$ path systems.
\end{proof}

\begin{maincorollary} \label{maincorollary}
If $\mathcal{C}$ is a Cauchon diagram and $I\subseteq [m]$ and $J\subseteq[n]$ are two subsets of the same size, then $\det_q(M_\mathcal{C}[I,J])=0$ if and only if there does not exist a vertex-disjoint $(R_I,C_J)$-path system.
\end{maincorollary}
\begin{proof}

If there does not exist a vertex-disjoint $(R_I,C_J)$-path system, then by Theorem~\ref{gv} $\det_q(M_\mathcal{C}[I,J])$ is the empty sum and so $\det_q(M_\mathcal{C}[I,J])=0$.

Conversely, suppose that there exists at least one vertex-disjoint $(R_I,C_J)$-path system. If $\mathcal{P}$ is one, then $w(\mathcal{P})$ consists of a nonempty sum of elements of the algebra $\mathcal{B}$ where each summand is a product of $\mathcal{B}$-generators and their inverses. By arranging each such product so the generators appear from left to right in lexicographic order, it follows that we can uniquely write
\begin{eqnarray}
{\det}_q(M_\mathcal{C}[I,J]) & = & \sum_{\mathcal{P}} w(\mathcal{P}) \nonumber \\
 & = & \sum_{Q\subseteq [m]\times [n]} P_Q(q)\prod_{\alpha\in Q} t_\alpha^{r(\alpha,Q)}, \label{equ}
\end{eqnarray} where $P_Q(q)$ is some polynomial in $\mathbb{Z}_{\geq 0}[q,q^{-1}]$, and $r(\alpha,Q)$ is an integer. 

Since at least one vertex-disjoint $(R_I,C_J)$-path system exists, the sum in Equation~\ref{equ} is non-empty and so there exists at least one $Q\subseteq [m]\times [n]$ such that $P_Q\not\equiv 0$. Since $q$ is transcendental over $\mathbb{Q}$, we know that $P_Q(q)\neq 0$ for any $P_Q\not\equiv 0$. Thus $\det_q(M_\mathcal{C}[I,J])\neq 0$.

\end{proof}

In the above proof we use the assumption that $q$ is transcendental over $\mathbb{Q}$. We believe that the result should remain true under the weakened hypothesis that $q\in \mathbb{K}^*$ is a non-root of unity. We note, however, that the main results of Launois~\cite{launois2} depend on a result of Hodges and Levasseur~\cite{hl} that requires $q$ to be transcendental over $\mathbb{Q}$. Thus we have made no attempt to prove Theorem~\ref{maincorollary} in the case $q$ is a non-root of unity.

\section{Finding Vanishing Quantum Minors}

Launois \cite{launois3} originally proved the following result for $\mathbb{K}=\mathbb{C}$, but by results of Goodearl, Launois and Lenagan~\cite{gll} it suffices to set $\mathbb{K}$ to be any field of characteristic zero. 

\begin{launois1} \label{l1}
Let $q\in \mathbb{K}^*$ be transcendental over $\mathbb{Q}$. The $\mathcal{H}$-invariant prime ideals of $\mathcal{A}$ are generated by quantum minors of the matrix of canonical generators $X_\mathcal{A}$.
\end{launois1}

Launois \cite{launois2} has given an algorithm which takes as input the Cauchon diagram corresponding to an $\mathcal{H}$-invariant prime ideal $\mathcal{I}$, and outputs a matrix whose vanishing quantum minors correspond to quantum minors of $X_\mathcal{A}$ which generate $\mathcal{I}$. This algorithm is essentially Cauchon's Deleting-Derivations Algorithm run in reverse.

\begin{algorithm}\label{algorithm}
(Note that the entries of every matrix below are from the algebra $\mathcal{B}$ from Definition~\ref{talg}).
\begin{description}
\item[Input] A Cauchon diagram $\mathcal{C}$.
\item[Output] A matrix $T^{(m,n)}$ with entries from the algebra $\mathcal{B}$.
	\item[Initialization] Let $T^{(1,1)}$ be an $m\times n$ matrix defined by 
	
	\begin{displaymath}
T^{(1,1)}[i,j] = \left\{ \begin{array}{ll}
t_{i,j} & \textrm{if $(i,j)\in W_\mathcal{C}$,}\\
0 & \textrm{if $(i,j)\in B_\mathcal{C}$.}
\end{array} \right.
\end{displaymath} 

Set $(s,t)=(1,2)$ and let $T^{(s,t)^-}[i,j]:=t^{(s,t)^-}_{i,j}$.
\item[While $(s,t)\neq (m,n+1)$,] do the following:

\begin{enumerate}
\item Construct the matrix $T^{(s,t)}$, where $T^{(s,t)}[i,j]:=t^{(s,t)}_{i,j}$, by

	\begin{displaymath}
t^{(s,t)}_{i,j} = \left\{ \begin{array}{ll}
t^{(s,t)^-}_{i,j} + t^{(s,t)^-}_{i,s}(t_{s,t}^{(s,t)^-})^{-1}\ t_{r,j}^{(s,t)^-}& \textrm{if $(i,j)\leq (s-1,t-1)$ and $t_{s,t} \neq 0$,}\\
t_{i,j}^{(s,t)^-} & \textrm{otherwise.}
\end{array} \right.
\end{displaymath} 
\item Set $(s,t)=(s,t)^+$.
\end{enumerate}
\item[End while.]
\end{description}
\end{algorithm}

Notice that we have $t^{(s,t)}_{s,k}=t^{(1,1)}_{s,k}$ for all $k\in[n]$. Launois \cite{launois3} proved the following.
\begin{launois2} \label{launois2}
Let $\mathcal{I}$ be an $\mathcal{H}$-invariant prime ideal of $\mathcal{A}=\mathcal{O}_q(\mathcal{M}_{m,n}(\mathbb{K}))$. Let $\mathcal{C}$ be the Cauchon diagram associated to $\mathcal{I}$. Apply Algorithm~\ref{algorithm} to $\mathcal{C}$ to obtain the matrix $T^{(m,n)}$. If a square submatrix in $T^{(m,n)}$ has a vanishing quantum minor, then the corresponding quantum minor in the matrix $X_\mathcal{A}$ is a generator for $\mathcal{I}$. Furthermore, $\mathcal{I}$ is generated by all such quantum minors. 
\end{launois2}

On the other hand, we prove the following.
\begin{pathmatrix} \label{pathmatrix}
For a Cauchon diagram $\mathcal{C}$, the path matrix $M_\mathcal{C}$ is the same as the matrix obtained at the end of Algorithm~\ref{algorithm}.
\end{pathmatrix}

Before proving this lemma in its full generality, we apply Algorithm~\ref{algorithm} to a small Cauchon diagram, and compare the result with the corresponding path matrix.

\begin{example}\label{ex}
\begin{figure} 
\centering
\includegraphics[height=6cm]{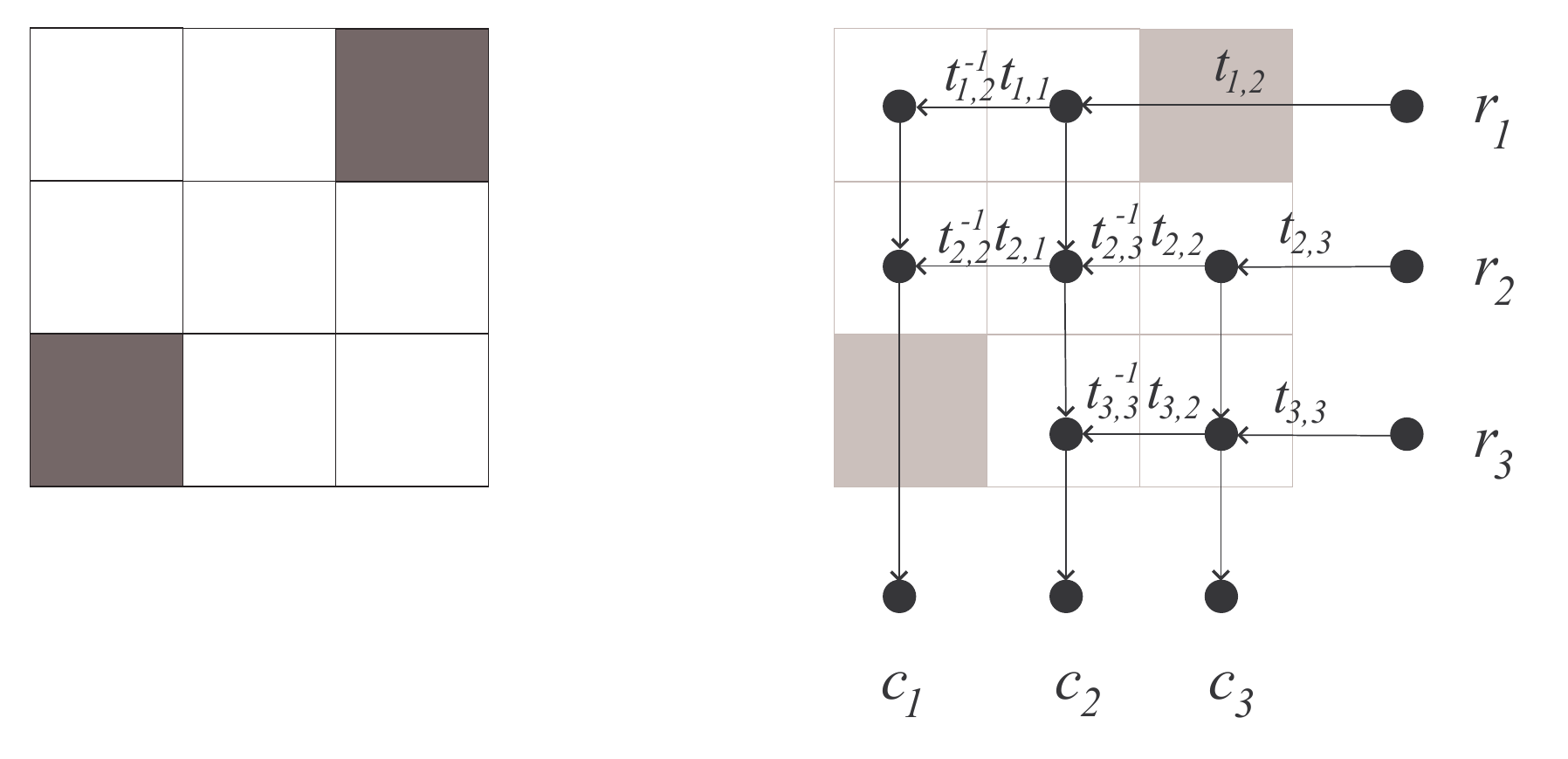}
\caption{Example~\ref{ex}: A Cauchon diagram on the left, and its Cauchon graph on the right (the unit weights on the vertical edges are not shown)}\label{ex2}
\end{figure}

Consider the $3\times 3$ Cauchon diagram $\mathcal{C}$ in Figure~\ref{ex2}. The initialization step of Launois' algorithm gives 

\begin{displaymath}
T^{(1,1)} =
\left[ \begin{array}{ccc}
t_{1,1} & t_{1,2} & 0 \\
t_{2,1} & t_{2,2} & t_{2,3} \\
0 & t_{3,2} & t_{3,3}
\end{array} \right].
\end{displaymath}

Notice that each non-zero entry $t_{i,j}$ is precisely the weight of the path $r_i\rightarrow (i,j)\rightarrow c_j$ in $\mathcal{G}_\mathcal{C}$.

Now at step $(s,t)$ of the algorithm, the only entries of $T^{(s,t)}$ that are modified from the previous step are those which are ``north-west'' of $(s,t)$. In particular, steps $(s,t)$ with either $s=1$ or $t=1$ do not change the previous matrix. In our example then, we have $T^{(1,1)}=T^{(1,2)}=T^{(1,3)}=T^{(2,1)}$. 

At step $(2,2)$, the only entry north-west of this entry is $(1,1)$. We therefore obtain

\begin{displaymath}
T^{(2,2)} =
\left[ \begin{array}{ccc}
t_{1,1}+ t_{1,2}t_{2,2}^{-1}t_{2,1} & t_{1,2} & 0 \\
t_{2,1} & t_{2,2} & t_{2,3} \\
0 & t_{3,2} & t_{3,3}
\end{array} \right].
\end{displaymath}
Notice that the new value of entry $(1,1)$ is precisely the weight of the path
path $r_1\rightarrow (1,1)\rightarrow c_1$ plus the weight of the path $r_1\rightarrow(1,2)\rightarrow(2,2)\rightarrow (2,1)\rightarrow c_1$.

The next step is $(2,3)$ and entry $(2,3)$ of $T^{(2,2)}$ is non-zero. However, the entries north-west of $(2,3)$ are $(1,1)$ and $(1,2)$, and since $T^{(2,2)}[1,3]=t^{(2,2)}_{1,3}=0$, the net effect of the algorithm at this step is to change nothing. Thus $T^{(2,3)}=T^{(2,2)}$. 

We also have $T^{(3,1)}=T^{(2,3)}$, so consider step $(3,2)$. By similar reasoning as in step $(2,3)$, we find $T^{(3,1)}=T^{(3,2)}$. The last step is $(3,3)$. Applying the algorithm we get

\begin{displaymath}
T^{(3,3)} =
\left[ \begin{array}{ccc}
t_{1,2}t_{2,2}^{-1}t_{2,1} & t_{1,2} & 0 \\
t_{2,1} & t_{2,2}+t_{2,3}t_{3,3}^{-1}t_{3,2} & t_{2,3} \\
0 & t_{3,2} & t_{3,3}
\end{array} \right].
\end{displaymath}

As one can easily verify, $T^{(3,3)}$ is precisely the path matrix $M_\mathcal{C}$. 

\end{example}

\begin{proof}[Proof of Lemma~\ref{pathmatrix}]

Fix $n$. We prove the lemma by induction on the number of rows $m$. As in Algorithm~\ref{algorithm}, we will denote $T^{(s,t)}[i,j]$ by $t^{(s,t)}_{i,j}\in \mathcal{B}$

First note that since we only modify entries which are north-west of the entry corresponding to the current step, the algorithm will always leave the $m$th row unmodified. That is, $T^{(s,t)}[m,[n]]=T^{(1,1)}[m,[n]]$ for all $(s,t)$. By the algorithm we get that, for $k\in [n]$,
	\begin{displaymath}
t^{(m,n)}_{m,k}=t^{(1,1)}_{m,k} = \left\{ \begin{array}{ll}
t_{m,k} & \textrm{if $(m,k)\in W_\mathcal{C}$,}\\
0 & \textrm{if $(m,k)\in B_\mathcal{C}$.}
\end{array} \right.
\end{displaymath}

Now in the $m$th row of $\mathcal{G}_\mathcal{C}$, there is clearly at most one possible path from $r_m$ to $c_k$. This path exists if and only if $(m,k)$ is a white square, and by Proposition~\ref{cgproperties}, this path has weight exactly $t_{m,k}$. 

From these two observations we see that the $m$th row in $T^{(m,n)}$ is exactly the same as the $m$th row in $M_\mathcal{C}$. Similarly, the $n$th column of $T^{(m,n)}$ is equal to the $n$th column of $M_\mathcal{C}$. In particular the lemma is true when $m=1$. 

Suppose that the lemma is true for all Cauchon diagrams with less than $m$ rows. If we obtain the Cauchon diagram $\mathcal{C}^\prime$ from $\mathcal{C}$ by deleting the $m$th row, then by induction we have $T^{(m-1,n)}\left[[m-1],[n]\right]=M_{\mathcal{C}^\prime}$. An equivalent way of stating this is that if $i<m$, then $t^{(m-1,n)}_{i,j}$ is the total of the weights of all paths in $\mathcal{G}_\mathcal{C}$ from $r_i$ to $c_j$ which \emph{do not} use a horizontal edge in row $m$.

As we already noted, $T^{(m,n)}[m,[n]]=M_\mathcal{C}[m,[n]]$ and $T^{(m,n)}[[m],n]=M_\mathcal{C}[[m],n]$ so to complete the proof, we establish the following claim by induction on $k\in[n]$ where $k$ will denote the $k$th column of $\mathcal{C}$. It will follow from this that $T^{(m,n)}=M_\mathcal{C}$. (At this point, the reader may wish to review the while loop in Algorithm~\ref{algorithm}).

\begin{claim}
If $(i,j)\leq(m-1,k-1)$, then $t^{(m,k)}_{i,j}$ is obtained from $t^{(m,k)^-}_{i,j}$ by adding the weights of all paths $P$ that satisfy the following properties: 
\begin{enumerate}
\item $P$ is a path from $r_i$ to $c_j$;
\item $P$ contains the subpath $K_j:(m,k)\Rightarrow c_j$. (Note that $K_j$ consists of horizontal edges from $(m,k)$ to $(m,j)$ and then the vertical edge $((m,j),c_j)$);
\item $P$ contains a vertical edge $((\ell,k),(m,k))$ for some $\ell<m$. In other words, if $k^\prime >k$, then vertex $(m,k^\prime)$ (if it exists) is \emph{not} an internal vertex of $P$.
\end{enumerate}
\end{claim}

\begin{figure} 
\centering
\includegraphics[height=6cm]{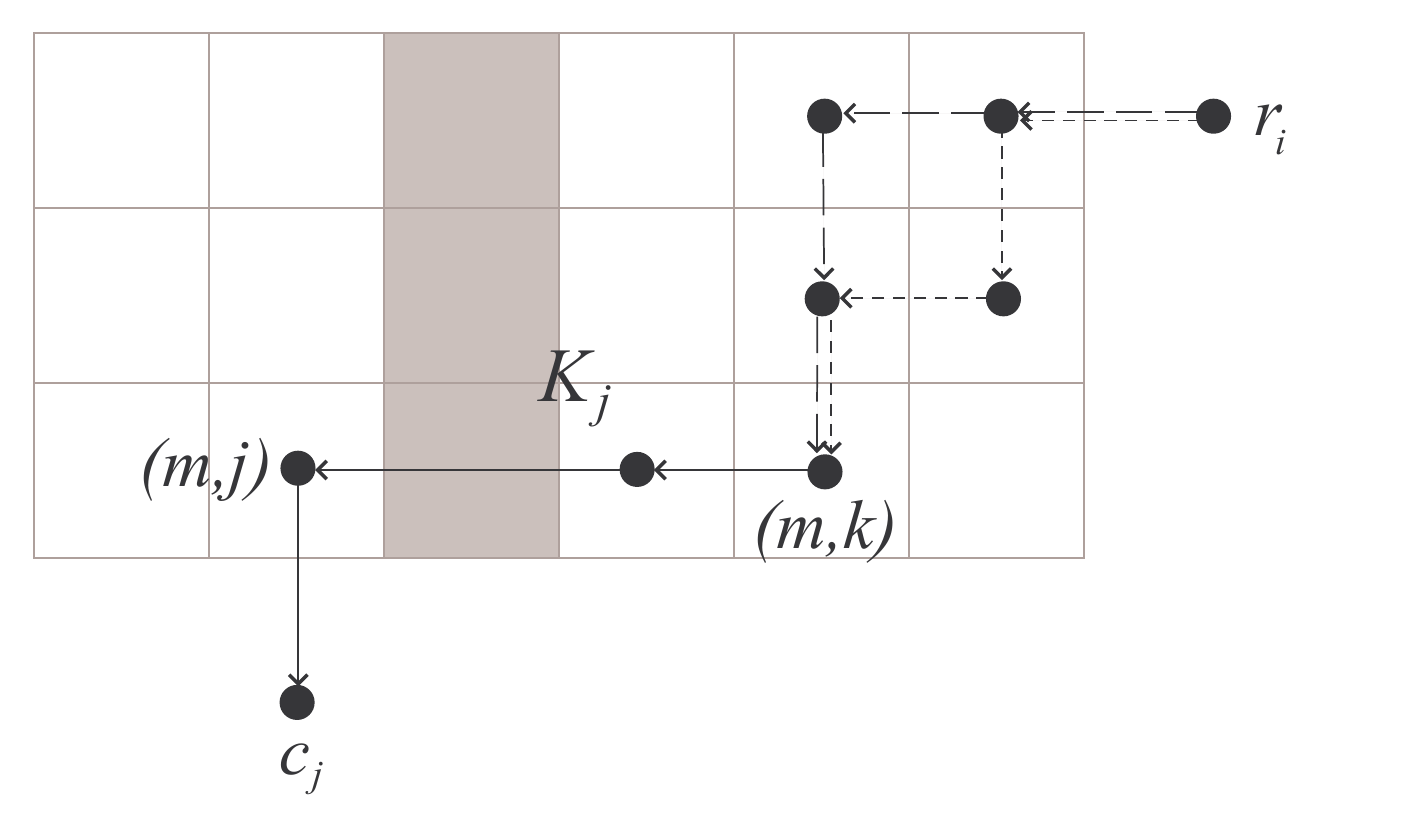}
\caption{$K_j$ is the path consisting of the edges drawn with a solid line. Concatenating $K_j$ with either of the paths drawn with dotted edges, gives a path which satisfies the properties in the claim.}\label{maintheoremfigure}
\end{figure}

For $k=1$, we know that $T^{(m,1)}=T^{(m-1,n)}$. On the other hand, since there are no $j<k$ the claim is trivially true. So let $k>1$ and assume that the claim is true for step $(m,k-1)$. 

If $(m,k)$ is black in $\mathcal{C}$, then according to Algorithm~\ref{algorithm}, we set $t^{(m,k)}_{i,j} =t^{(m,k-1)}_{i,j}$. On the other hand, if $(m,k)$ is black, then $K_j$ can not exist for any $j$ and so there are no paths which satisfy the properties in the claim. This then proves the claim in the case that $(m,k)$ is black.

Suppose that $(m,k)$ is white. If $(m,j)$ is black for some $j<k$, then $t^{(m,k-1)}_{m,j}=t^{(1,1)}_{m,j}=0$ and so by the algorithm we again have $t^{(m,k)}_{i,j} =t^{(m,k-1)}_{i,j}$. On the other hand, if $(m,j)$ is black, then $K_j$ can not exist. This proves the statement in the claim for those $j<k$ such that $(m,j)$ is black.

Finally, if $(m,k)$ and $(m,j)$ are white squares, then $K_j$ exists in $\mathcal{G}_\mathcal{C}$. Note that $w(K_j)= t_{m,k}^{-1}t_{m,j}$ by Proposition~\ref{cgproperties}. On the other hand, if $P$ is a path satisfying the properties in the claim, then there is a path $P^\prime:r_i\Rightarrow (m,k)$ such that $P=P^\prime K_j$. By Property 3, the last edge in $P^\prime$ is vertical. So if we concatenate $P^\prime$ with the vertical path $L_k:(m,k)\Rightarrow c_k$, then we get a path (with the same weight as $P^\prime$) from $r_i$ to $c_k$ which does not use any horizontal edge in the last row. 

By induction, the set of all such $P^\prime$ has total weight $t^{(m-1,k)}_{i,k}$. But this entry has not been modified at step $(m,\ell)$ of the algorithm for any $\ell<k$, so in fact, the set of all such $P^\prime$ has total weight $t^{(m,k-1)}_{i,k}$. Hence the total weight of all $P$ that satisfy the properties in the claim is exactly $$t^{(m,k-1)}_{i,k}w(K_j) = t^{(m,k-1)}_{i,k}t_{m,k}^{-1}t_{m,j}.$$ On the other hand, by Algorithm~\ref{algorithm}, $$t^{(m,k)}_{i,j}=t^{(m,k-1)}_{i,j}+ t^{(m,k-1)}_{i,k}t_{m,k}^{-1}t_{m,j}.$$ This finishes the proof of the claim and the lemma.

\end{proof}

Now we state the main result of this paper, which follows immediately from Theorem~\ref{maincorollary}, Theorem~\ref{launois2} and Lemma~\ref{pathmatrix}. 

\begin{maintheorem} \label{main}
Let $\mathcal{C}$ be an $m\times n$ Cauchon diagram corresponding to the $\mathcal{H}$-invariant prime ideal $\mathcal{I}$. A quantum minor ${\det}_q (X_\mathcal{A}[I,J])$ of $X_\mathcal{A}$ is in $\mathcal{I}$ if and only if there does not exist a vertex-disjoint $(R_I,C_J)$-path system in the Cauchon graph $\mathcal{G}_\mathcal{C}$.
\end{maintheorem}

We should emphasize that, in general, the preceding theorem does not give a \emph{minimal} generating set; further work must be done to find one. It is unclear if Yakimov's method~\cite{yakimov} produces a minimal generating set. 

\section{Concluding Remarks}

We note that Algorithm~\ref{algorithm} can, in general, result in a matrix which has entries with exponentially many terms. This algorithm is therefore not always ideal if one simply wishes to check whether a specific quantum minor appears in the generating set given by Theorem~\ref{l1}. On the other hand, finding a vertex-disjoint path system in a Cauchon graph is computationally efficient as it is a special case of Menger's Theorem in graph theory, which is well-known to be solvable in polynomial time~\cite{menger}.

\section*{Acknowledgments}
The author would like to thank Jason Bell, St\'ephane Launois and, most especially, the anonymous referee for valuable suggestions which improved the presentation of this paper. 

\bibliography{casteels1}
\bibliographystyle{amsplain}
\end{document}